\documentclass[12pt, a4paper]{article}
\usepackage{latexsym}
\usepackage{amssymb}
\usepackage{amsfonts}
\usepackage{amsthm}
\usepackage{amsmath}
\usepackage[pdftex]{graphicx}
\usepackage{comment}
\usepackage{epstopdf}
\usepackage{epsfig}
\usepackage{float}

\newcommand{\eps}{{\varepsilon}}
\newcommand{\argdot}{\, . \,}

\newtheorem{theorem}{Theorem}[section]

\newtheorem{proposition}[theorem]{Proposition}

\newtheorem{definition}[theorem]{Definition}

\newtheorem{remark}[theorem]{Remark}

\newtheorem{algorithm}[theorem]{Algorithm}

\renewcommand{\theequation}{\arabic{section}.\arabic{equation}}

\begin{document}
\title{A variation of the Canadisation algorithm for the pricing of American options driven by L\'evy processes}
\author{Florian Kleinert\footnote{School of Mathematics, University of Manchester, {\sc Manchester, M13 9PL, United Kingdom}. E-mail: florian.kleinert@manchester.ac.uk} \and Kees van Schaik\footnote{School of Mathematics, University of Manchester, {\sc Manchester, M13 9PL, United Kingdom}. E-mail: kees.vanschaik@manchester.ac.uk}}

\date{}
\maketitle

\begin{abstract} 

We introduce an algorithm for the pricing of finite expiry American options driven by L\'evy processes. The idea is to tweak Carr's `Canadisation' method, cf. Carr \cite{Carr98} (see also Bouchard et al \cite{Bouchard05}), in such a way that the adjusted algorithm is viable for any L\'evy process whose law at an independent, exponentially distributed time consists of a (possibly infinite) mixture of exponentials. This includes Brownian motion plus (hyper)exponential jumps, but also the recently introduced rich class of so-called meromorphic L\'evy processes, cf. Kyprianou et al \cite{Kuznetsov10}. This class contains all L\'evy processes whose L\'evy measure is an infinite mixture of exponentials which can generate both finite and infinite jump activity. L\'evy processes well known in mathematical finance can in a straightforward way be obtained as a limit of meromorphic L\'evy processes. We work out the algorithm in detail for the classic example of the American put, and we illustrate the results with some numerics.
\end{abstract}

\noindent
{\footnotesize L{{\'e}}vy processes}



\section{Introduction}
\setcounter{equation}{0}

Let $X$ be a L\'evy process defined on a filtered
probability space $(\Omega,\mathcal{F},\mathbf{F},\mathbb{P})$, where $\mathbf{F}=(\mathcal{F}_t)_{t \geq 0}$ is the filtration generated by $X$ which is naturally enlarged (cf. Definition 1.3.38 in \cite{Bichteler02}). Recall that a L\'evy process is characterised by stationary, independent increments and paths which are right continuous and have left limits, and its law is characterised by the characteristic exponent $\Psi$ defined through $\mathbb{E}[e^{\mathrm{i} z X_t}]=e^{-t \Psi(z)}$ for all $t \geq 0$ and $z \in \mathbb{R}$. According to the L\'evy-Khintchine formula there exist $\sigma \geq 0$, $a \in \mathbb{R}$ and a measure $\Pi$ (the L\'evy measure) concentrated on $\mathbb{R} \setminus \{0\}$ satisfying $\int_{\mathbb{R}} (1 \wedge x^2) \, \Pi(dx)<\infty$ such that

\[ \Psi(z) = \frac{\sigma^2}{2} z^2 + \mathrm{i} az + \int_{\mathbb{R}} \left( 1-e^{\mathrm{i}zx}+\mathbf{1}_{\{ |x|<1 \}} \mathrm{i} zx \right) \, \Pi(dx) \]
for all $z \in \mathbb{R}$. The tuple $(\sigma,a,\Pi)$ is usually refered to as the L\'evy triplet. For a further introduction see e.g. the textbooks \cite{Bertoin96, Kyprianou06, Sato99}.

For $x \in \mathbb{R}$ we denote by $\mathbb{P}_x$ the law of $X$ when it is started at $x$ and we abbreviate $\mathbb{P}=\mathbb{P}_0$. Accordingly we shall write $\mathbb{E}_x$ and $\mathbb{E}$ for the
associated expectation operators. We denote by $\mathcal{T}$ the set of $\mathbf{F}$-stopping times.

Consider the following classic optimal stopping problem with expiry date $T \in [0,\infty)$:

\begin{equation}\label{def_stop_orig} 
v(T,x) = \sup_{\tau \in \mathcal{T}} \mathbb{E}_x \left[ e^{-r(\tau \wedge T)} f(X_{\tau \wedge T}) \right],
\end{equation}
where $f$ is a payoff function and $r \geq 0$ the discount rate. Such problems have been studied thoroughly and have applications in many fields, where option pricing theory in modern mathematical finance is maybe one of the most appealing ones. See e.g. the textbook \cite{Peskir06} for a recent overview of the general theory concerning optimal stopping problems, and e.g. \cite{Karat01} for an introduction to mathematical finance and explanation of the related terms used here. For example, if in a typical Black \& Scholes setup the price of the risky asset $S$ is assumed to evolve as $S_t=\exp(X_t)$ for all $t \geq 0$ while the price of the riskless asset $S^0$ is given by $S^0_t = e^{rt}$ for all $t \geq 0$ then $v$ can be interpreted\footnote{In addition it is required that $S/S^0$ is a martingale. If $X$ is not a Brownian motion there are typically multiple measures $Q$ equivalent to each other under which $S/S^0$ is a martingale. For this paper it is not relevant which measure is used, as long as $X$ is a L\'evy process under the chosen measure and $\mathbb{P}$ is understood to be the chosen measure} as the fair price of an American option on the risky asset which expires after $T$ time units and $x=X_0=\log S_0$. One of the classic examples is the American put option, which is characterised by the payoff function $f(x)=(K-e^x)^+$ for some $K>0$. Here $a^+:=\max\{a,0\}$. Such L\'evy process driven market models have received considerable attention in recent years, in an attempt to overcome some of the limitations of the classic Black \& Scholes model where $X$ is a Brownian motion. See e.g. \cite{Cont04} and the references therein.

It is well known that in general no closed form formula exists for (\ref{def_stop_orig}). Results do exist for certain combinations of payoff functions $f$ and L\'evy processes $X$ when $T=\infty$, but when $T<\infty$ --- as we assume throughout this paper --- one typically has to settle for approximating $v$ rather. 

In his celebrated paper \cite{Carr98} (see also \cite{Bouchard05}), Carr proposes a method for this he refers to as `Canadisation'. The idea is to introduce a `stochastic time grid', which is a grid where the distances between consecutive grid points forms a sequence of independent random variables with an identical exponential distribution. The grid gets finer as the common mean of the exponential distributions vanishes. The value of an American option with expiry date $T$ is then approximated by the value of the same American option but with the expiry date replaced by a grid point close to $T$. This approximating value function can be determined by performing a backwards induction over the grid points. Carr works out this algorithm for the American put option driven by a Brownian motion. See also \cite{Boyarchenko05} and the references therein for more general applications of Carr's algorithm.

As shown in \cite{Boyarchenko05} Carr's Canadisation is also viable for Brownian motion with exponential jumps. However, it does not seem easy to apply this algorithm to more general L\'evy processes due to the fact that the backwards induction over the grid points is involved and (hence) imposes restrictions. This paper introduces a variation of the Canadisation algorithm, where some changes to the original setup are made so that the backwards induction over the grid points becomes more straightforward. Indeed, the resulting adjusted algorithm allows to approximate $v$ by elementary functions not just for Brownian motion plus exponential jumps but for any so-called \emph{meromorphic} L\'evy process.  

The class of meromorphic L\'evy processes was recently introduced in \cite{Kuznetsov10} and can be defined as the family of all L\'evy processes whose L\'evy measure has a density that consists of an infinite mixture of exponentials, see Section \ref{sec_example} for further details and examples of such processes. This is a rich class, where in particular also infinite activity jumps are possible. It contains for instance the family of so-called $\beta$-processes (cf. \cite{Kuznetsov09}). Furthermore many L\'evy processes that play a prominent role in mathematical finance can in a straightforward way be obtained by considering limiting cases of meromorphic L\'evy processes, see \cite{Kuznetsov09} and \cite{Kuznetsov10}. Examples include generalized tempered stable processes (cf. \cite{Cont04}), Kobol processes (cf. \cite{Boyarchenko00} and \cite{boyarchenko02barrier}), tempered stable processes (cf. \cite{barndorff01}) and CGMY processes (cf. \cite{carr02}). 

In particular meromorphic L\'evy processes have the property that the law of the process evaluated at an independent, exponentially distributed time is an infinite mixture of exponentials (cf. Proposition \ref{prop_MProc}). (In fact meromorphic L\'evy processes are the only L\'evy processes for which this is the case.) Note that for a Brownian motion with hyperexponential jumps this law is a finite mixture of exponentials. This property is key in our algorithm, as becomes clear in Section \ref{sec_example}. Indeed, if this is the case every step in the algorithm can be worked out explicitly in terms of elementary functions provided the payoff function $f$ in (\ref{def_stop_orig}) can (piecewise) be expressed as a linear combination of functions of the form $A x^i e^{B x} + C$ for $A, B, C \in \mathbb{R}$ and $i \in \mathbb{N}$. This is for instance the case for the classic example of the American put option, where $f(x)=(K-e^x)^+$ for some $K>0$. For any payoff function not of this type it is a straightforward exercise to approximate it by a function of this type and estimate the error due to this approximation, cf. Remark \ref{rem1}.

Pricing American options driven by L\'evy processes has received considerable attention in recent years. Indeed several other possible techniques were developed, including (but not limited to) tree-based methods (see e.g. \cite{Maller06}) and variational methods (see e.g. \cite{matache2005}). See also the overview in \cite{levendorskii2006}. A detailed comparison of all available methods with the one proposed in this paper would constitute a paper in itself and we do not embark on such a task here. The authors feel the main appeal of the method presented in this paper lies in the fact that it is a `light weight' method which is probabilistic in nature and does not need any involved proofs or tools from other fields. Furthermore the computer implementation is rather straightforward as only elementary operations are required (together with the straightforward operation of numerical root finding). Finally the method is flexible and can easily be adjusted to deal with situations where for instance the option has a Bermudan type structure, i.e. where exercising is only allowed in certain subsets of $[0,T]$, or there is some path dependency in the payoff (for instance when the payoff also depends on the path supremum). Cf. Remark \ref{rem1}.

The rest of this paper is organised as follows. In Section \ref{sec_alg} we introduce the algorithm in detail and spend some time discussing the difference with the original Canadisation approach. In Section \ref{sec_Main} we make the algorithm rigorous. In Section \ref{sec_example} we work out the algorithm in detail for the prominent example of the American put, where the driving process is a meromorphic L\'evy process (or a Brownian motion plus (hyper)exponential jumps). In that section we also provide some more details concerning meromorphic L\'evy processes. We find that the functions approximating $v$ generated by our algorithm have easy to implement, explicit formulae. We conclude by presenting some numerical results in Section \ref{sec_num}.

\section{The algorithm and comparison with Canadisation}\label{sec_alg}

Let us introduce the algorithm in detail. For any $n \in \mathbb{N}$, enlarge the above probability space to contain a Poisson process $N^{(n)}$ with intensity $n$, independent of $X$. Denote the $k$-th jump time of $N^{(n)}$ by $T^{(n)}_k$, i.e.

\begin{equation}\label{def_T^n_k} 
T^{(n)}_k := \inf \{ t \geq 0 \, | \, N^{(n)}_t \geq k \} 
\end{equation}
 with $T^{(n)}_0:=0$. Let $\widetilde{\mathbf{F}}^{(n)}$ denote the naturally enlarged filtration generated by the pair $(X,N^{(n)})$. Note that for any $T \geq 0$, if $(k(n))_{n \geq 1}$ is a sequence of natural numbers such that $k(n)/n \to T$ as $n \to \infty$ then $T^{(n)}_{k(n)} \to T$ a.s. by the law of large numbers.

Furthermore, denote by $\widetilde{\mathcal{T}}^{(n)}$ the set of $\widetilde{\mathbf{F}}^{(n)}$-stopping times on this enlarged probability space that \emph{only take values on the grid points} $\{ 0=T^{(n)}_0, T^{(n)}_1, \ldots \}$, that is

\[ \widetilde{\mathcal{T}}^{(n)} := \left\{ \tau \, \left| \, \tau \mbox{ is an $\widetilde{\mathbf{F}}^{(n)}$-stopping time \& $\tau \in \{ 0=T^{(n)}_0, T^{(n)}_1, \ldots \}$} \right. \right\}. \]

Now consider modifying the original optimal stopping problem (\ref{def_stop_orig}) in three steps. First, replace the deterministic expiry date $T$ by the random variable $T^{(n)}_k$ (for suitably chosen $k$). Second, replace the set of stopping times $\mathcal{T}$ over which is optimized by $\widetilde{\mathcal{T}}^{(n)}$. Finally, replace the discount factor $e^{-r\tau}$ by $D^{(n)}(\tau)$ defined as

\begin{equation}\label{def_discnt}
D^{(n)}(\tau) := \sum_{i=0}^{\infty} \mathbf{1}_{\{ \tau=T^{(n)}_i \}} e^{-ri/n}.
\end{equation}
Together this amounts to defining for each $n \in \mathbb{N}$ and $k \in \mathbb{N}$:

\begin{equation}\label{def_stop_new} 
V_k^{(n)}(x) := \sup_{\tau \in \widetilde{\mathcal{T}}^{(n)}} \mathbb{E}_x \left[ D^{(n)}(\tau \wedge T^{(n)}_k) f(X_{\tau \wedge T^{(n)}_k}) \right].
\end{equation}

The usefulness of this setup relies on the following two facts. For any expiry date $T \geq 0$, choosing a sequence $(k(n))_{n \geq 1}$ such that $k(n)/n \to T$ as $n \to \infty$ we have for any $x \in \mathbb{R}$ that $V_{k(n)}^{(n)}(x) \to v(T,x)$ as $n \to \infty$ (cf. Theorem \ref{thm_main} (ii)). Furthermore, for any $n \in \mathbb{N}$ the sequence of functions $(V_k^{(n)})_{k \geq 0}$ satisfies the following recursion:

\begin{equation}\label{eq_recursion} 
V_0^{(n)}(x) = f(x), \quad V_k^{(n)}(x)=\max \left\{ f(x), e^{-r/n} \mathbb{E}_x \left[  V_{k-1}^{(n)}(X_{\xi^{(n)}}) \right] \right\} \mbox{ for $k \geq 1$,} 
\end{equation}
where $\xi^{(n)}$ is an exponentially distributed random variable with mean $1/n$, independent of $X$. This is easily seen by a dynamic programming argument. Indeed, consider (\ref{def_stop_new}) for some $x \in \mathbb{R}$. The available stopping times allow to stop either immediately or at $T^{(n)}_1$ or later. Stopping immediately yields a payoff of $f(x)$ while the (current) value of waiting until $T^{(n)}_1$ or later equals the discounted expected value of the option with expiry date $T^{(n)}_{k-1}$ for $x=X_{T^{(n)}_1}$. Hence $V_k^{(n)}(x)$ equals the maximum of the two.

The main point of this setup is that the recursion (\ref{eq_recursion}) is very straightforward. Indeed, as already alluded to in the Introduction, for any meromorphic L\'evy process the law of $X_{\xi^{(n)}}$ is a mixture of exponentials. As a consequence (\ref{eq_recursion}) can be worked out explicitly for example if $f(x)=(K-e^x)^+$, where first the expectation is computed explicitly and next the maximum is easily determined, cf. Section \ref{sec_example}.

Now let us for comparison recall the original Canadisation algorithm introduced by Carr. With $T^{(n)}_k$ as defined in (\ref{def_T^n_k}), define for all $n \in \mathbb{N}$ and $k \in \mathbb{N}$

\[ \widehat{V}^{(n)}_k(x) = \sup_{\tau} \mathbb{E}_x \left[ e^{-r(\tau \wedge T^{(n)}_k)} f(X_{\tau \wedge T^{(n)}_k}) \right], \]
where the supremum is taken over all $\widetilde{\mathbf{F}}^{(n)}$-stopping times. It can be shown (cf. also \cite{Bouchard05}) that this yields the following recursion:

\begin{eqnarray} 
 \widehat{V}^{(n)}_k(x) & = & \sup_{\tau \in \mathcal{T}} \mathbb{E}_x \left[ \mathbf{1}_{\{ \tau<\xi^{(n)} \}} e^{-r\tau} f(X_{\tau} ) + \mathbf{1}_{\{ \tau \geq \xi^{(n)} \}} e^{-r\xi^{(n)}}  \widehat{V}^{(n)}_{k-1}(X_{\xi^{(n)}}) \right] \nonumber\\
 & = & \sup_{\tau \in \mathcal{T}} \mathbb{E}_x \left[  e^{-(r+n)\tau} f(X_{\tau} ) + n \int_0^{\tau} e^{-(r+n)u}  \widehat{V}^{(n)}_{k-1}(X_{u}) \, \mbox{d}u \right]. \label{Carr1}
\end{eqnarray} 
To work out this recursion, general theory of optimal stopping (cf. e.g. \cite{Peskir06}) dictates that the optimal stopping time in (\ref{Carr1}) is the first time that $X$ enters a subset $\mathcal{S}$ of the state space $\mathbb{R}$, say $\tau_{\mathcal{S}}$. Hence, in order to find an explicit expression for (\ref{Carr1}) one requirement is to know the mean of $e^{-(r+n)\tau_{\mathcal{S}}} f(X_{\tau_{\mathcal{S}}} )$. Carr worked with Brownian motion in which case this quantity is available in closed form. However once $X$ has jumps this becomes much more involved. For meromorphic L\'evy processes expressions are available that could in principle be used when $\mathcal{S}$ is a halfline or an interval (cf. Theorem 3 and Theorem 5 in \cite{Kuznetsov10}) but these are quite involved and complicated to implement. If $\mathcal{S}$ consists of a union of intervals the situation becomes even much more complicated. Namely, for the overshoot the possibility that $X$ jumps between different disconnected parts of $\mathcal{S}^c$ (the complement of $\mathcal{S}$) would have to be taken into account. That is, for some $x \in \mathcal{S}^c$ the rhs of (\ref{Carr1}) would depend on the values the function $\widehat{V}^{(n)}_k$ takes in intervals that are part of $\mathcal{S}^c$ and that $X$ could reach by a jump without entering $\mathcal{S}$ first.

Furthermore, in the case of an American put, i.e. $f(x)=(K-e^x)^+$, it can be shown that $\mathcal{S}=(-\infty,x^*]$ for some $x^*$ (depending on $n$ and $k$) and it remains to determine $x^*$. For other payoff functions it may a priori very well not be clear at all what shape $\mathcal{S}$ exactly takes, which makes the above task even more challenging.

On the other hand, the recursion (\ref{eq_recursion}) generated by our algorithm is much more straightforward to implement. In the case of the American put it can a priori be shown that a level $x^*$ (depending on $n$ and $k$) exists such that (cf. Proposition \ref{prop_alg_put})

\[ V_k^{(n)}(x) = \begin{cases}
f(x) & \text{if $x \leq x^*$} \\
e^{-r/n} \mathbb{E}_x \left[  V_{k-1}^{(n)}(X_{\xi^{(n)}}) \right] & \text{if $x > x^*$}
\end{cases} \]
and working out this recursion is easily possible for any meromorphic L\'evy process. For other payoff functions where a priori it is not clear how the maximum in (\ref{eq_recursion}) works out it is still rather straightforward to let a computer do the work. The first step would be to determine an explicit formula for 

\begin{equation}\label{K1} 
x \mapsto e^{-r/n} \mathbb{E}_x \left[  V_{k-1}^{(n)}(X_{\xi^{(n)}}) \right]
\end{equation}
and then $V_k^{(n)}(x)$ is the maximum of (\ref{K1}) and $f(x)$. With formulae available for both (\ref{K1}) and $f$ this is rather straightforward to implement using a computer. See also Remark \ref{rem1} (ii).

\begin{remark} A natural question is whether it would not be better to use a classic deterministic grid rather than a stochastic grid. That is to say, setting $t_k=kT/n$ for $k=0,1,\ldots,n$ one could consider approximating $v$ by narrowing the set of stopping times over which is optimised to the set of stopping times taking values in $\{t_0,t_1,\ldots,t_n\}$ only. Denoting the resulting value function with expiry date $t_k$ by $U^{(n)}_k$ the sequence of functions $(U^{(n)}_k)_{k=0,1,\ldots,n}$ is determined by the recursion

\begin{equation}\label{eq_rec2} 
U_0^{(n)}(x) = f(x), \quad U_k^{(n)}(x)=\max \left\{ f(x), e^{-r/n} \mathbb{E}_x \left[ U_{k-1}^{(n)}(X_{1/n}) \right] \right\} \mbox{ for $k=1,\ldots,n$.} 
\end{equation}
The problem is that this recursion involves the law of $X_{1/n}$. Even if this is known (which generally is not the case except for a few notable exceptions) it is typically not of a friendly enough nature to work out the recursion (\ref{eq_rec2}) explicitly. Indeed consider for instance a Brownian motion.
\end{remark}

\section{Main result \& proof}\label{sec_Main}
\setcounter{equation}{0}

This section is dedicated to making the algorithm outlined in the above Section \ref{sec_alg} rigorous.

\begin{theorem}\label{thm_main} Let $f$ be a bounded and continuous function. Let the value function $v$ be given by (\ref{def_stop_orig}). For each $n \geq 1$ and $k \geq 0$ let the function $V^{(n)}_k$ be given by (\ref{def_stop_new}). We have the following.
\begin{itemize}
\item[(i)] For any $n \geq 1$, the sequence of functions $(V_k^{(n)})_{k \geq 0}$ satisfies the following recursion:

\begin{equation}\label{main_rec} 
V_0^{(n)}(x) = f(x), \quad V_k^{(n)}(x)=\max \left\{ f(x), e^{-r/n} \mathbb{E}_x \left[  V_{k-1}^{(n)}(X_{\xi^{(n)}}) \right] \right\} \mbox{ for $k \geq 1$,} 
\end{equation}
where $\xi^{(n)}$ is independent of $X$ and follows an exponential distribution with mean $1/n$.
\item[(ii)] For any $T \in [0,\infty)$, if $(k(n))_{n \geq 1}$ is a sequence such that $k(n)/n \to T$ as $n \to \infty$ then:

\[ T^{(n)}_{k(n)} \longrightarrow T \mbox{ a.s.} \quad \mbox{and} \quad V_{k(n)}^{(n)}(x) \longrightarrow v(T,x) \mbox{ for all $x \in \mathbb{R}$.} \]
\end{itemize}
\end{theorem}

\begin{proof} Ad (i). This is just the dynamic programming principle applied to (\ref{def_stop_new}). Indeed for any $k \geq 1$

\begin{eqnarray*}
V_k^{(n)}(x) & = & \sup_{\tau \in \widetilde{\mathcal{T}}^{(n)}} \mathbb{E}_x \left[ \left( \mathbf{1}_{\{ \tau=0 \}} + \mathbf{1}_{\{ \tau \geq T^{(n)}_1 \}} \right) D^{(n)}(\tau \wedge T^{(n)}_k) f(X_{\tau \wedge T^{(n)}_k}) \right] \nonumber\\
 & = & \sup_{\tau \in \widetilde{\mathcal{T}}^{(n)}} \mathbb{E}_x \left[  \mathbf{1}_{\{ \tau=0 \}} f(x) + \mathbf{1}_{\{ \tau \geq T^{(n)}_1 \}}  \mathbb{E}_x \left[ \left. D^{(n)}(\tau \wedge T^{(n)}_k) f(X_{\tau \wedge T^{(n)}_k}) \, \right| \, \mathcal{F}_{T^{(n)}_1} \right] \right] \nonumber\\
 & = & \sup_{\tau \in \widetilde{\mathcal{T}}^{(n)}} \mathbb{E}_x \left[  \mathbf{1}_{\{ \tau=0 \}} f(x) + \mathbf{1}_{\{ \tau \geq T^{(n)}_1 \}} e^{-r/n}  \mathbb{E}_x \left[  V_{k-1}^{(n)}(X_{\xi^{(n)}}) \right] \right] \nonumber\\
 & = & \max \left\{ f(x), e^{-r/n} \mathbb{E}_x \left[  V_{k-1}^{(n)}(X_{\xi^{(n)}}) \right] \right\}. \nonumber
\end{eqnarray*}

Ad (ii). $T^{(n)}_{k(n)} \rightarrow T$ a.s. as $n \to \infty$ is just an application of the law of the large numbers. For proving that $V_{k(n)}^{(n)}(x) \to v(T,x)$ as $n \to \infty$ for any $x \in \mathbb{R}$ we have to do a bit more work. In this proof we will also need an auxiliary sequence of functions that is generated by our algorithm if we leave out the step of adjusting the discounting factor, that is, let for $n \geq 1, k \geq 0$ the function $\hat{V}^{(n)}_k$ be given by (compare with (\ref{def_stop_new})):

\begin{equation}\label{def_hatV}
\hat{V}_k^{(n)}(x) = \sup_{\tau \in \widetilde{\mathcal{T}}^{(n)}} \mathbb{E}_x \left[ e^{-r(\tau \wedge T^{(n)}_k)} f(X_{\tau \wedge T^{(n)}_k}) \right]. 
\end{equation}
Take any $x \in \mathbb{R}$. The result follows from the following three steps.\newline
Step 1:
\begin{equation}\label{res_step1}
\limsup_{n \to \infty} \hat{V}_{k(n)}^{(n)}(x) \leq v(T,x). 
\end{equation}
Step 2:
\begin{equation}\label{res_step2}
\liminf_{n \to \infty} \hat{V}_{k(n)}^{(n)}(x) \geq v(T,x).
\end{equation}
Step 3:
\begin{equation}\label{res_step3}
\lim_{n \to \infty} \left| \hat{V}_{k(n)}^{(n)}(x) - V_{k(n)}^{(n)}(x) \right| =0.
\end{equation}

\emph{Step 1.} We appeal to Carr's Canadisation as outlined in Section \ref{sec_alg}, however we follow the setup from \cite{Bouchard05}. Define for $n \geq 1$ the sequence of functions $(U^{(n)}_k)_{k \geq 0}$ as follows (recall that $\mathcal{T}$ is the set of stopping times with respect to the completed filtration generated by $X$):

\begin{equation}\label{27oct2}
 U_0^{(n)}(x) = f(x), \quad U_k^{(n)}(x)= \sup_{\tau \in \mathcal{T}} \mathbb{E}_x \left[ e^{-(r+n)\tau} f(X_{\tau}) + n \int_0^{\tau} e^{-(r+n)u} U_{k-1}^{(n)}(X_u) \, \mbox{d}u \right] \mbox{ for $k \geq 1$.} 
\end{equation}
It is shown in \cite{Bouchard05} that 

\begin{equation}\label{conv_ElKar} 
U_{k(n)}^{(n)}(x) \to v(T,x) \quad \mbox{as $n \to \infty$.} 
\end{equation}
By a dynamic programming argument the sequence $(\hat{V}_{k}^{(n)})_{k \geq 0}$ as defined in (\ref{def_hatV}) satisfies the recursion

\[ \hat{V}_0^{(n)}(x) = f(x), \quad \hat{V}_{k}^{(n)}(x) = \max \left\{ f(x), \mathbb{E}_x \left[ e^{-r \xi^{(n)}} \hat{V}_{k-1}^{(n)}(X_{\xi^{(n)}}) \right] \right\} \mbox{ for $k \geq 1$,} \]
where $\xi^{(n)}$ is an exponentially distributed random variable with mean $1/n$, independent of $X$. Comparing this recursion with (\ref{27oct2}), we see that $\hat{V}_{k-1}^{(n)} \leq U_{k-1}^{(n)}$ implies $\hat{V}_{k}^{(n)} \leq U_{k}^{(n)}$, since in (\ref{27oct2}) the supremum is at least as large as taking either $\tau=0$ or $\tau=\infty$. Since $\hat{V}_0^{(n)}=U_0^{(n)}=f$ it follows that $\hat{V}_{k}^{(n)} \leq U_{k}^{(n)}$ for all $k,n$, which together with (\ref{conv_ElKar}) indeed yields (\ref{res_step1}).

\emph{Step 2.} Let $\eps>0$ be arbitrary and let us show that

\begin{equation}\label{tmp4}
\liminf_{n \to \infty} \hat{V}_{k(n)}^{(n)}(x) \geq v(T,x)-\eps.
\end{equation}
For this, first note that since $X$ is continuous over stopping times there exists a $\tau'$ taking values in $0=t_0<t_1<\ldots<t_N=T$ such that

\begin{equation}\label{tmp1} 
\mathbb{E}_x \left[ e^{-r\tau' } f(X_{\tau' }) \right] \geq v(T,x)-\eps.
\end{equation}
Define an approximating sequence of stopping times for $\tau'$ by setting for any $n \geq 1$

\[ \widetilde{\mathcal{T}}^{(n)} \ni \sigma^{(n)} := \inf \left\{  T^{(n)}_k \, | \, T^{(n)}_k \geq \tau' \right\}. \]
Now we claim that as $n \to \infty$:

\begin{equation}\label{tmp2}
\mathbb{E}_x \left[ e^{-r\sigma^{(n)}} f(X_{\sigma^{(n)}}) \right] \longrightarrow \mathbb{E}_x \left[ e^{-r\tau'} f(X_{\tau'}) \right]
\end{equation}
and

\begin{equation}\label{tmp3}
\mathbb{E}_x \left[ \left| e^{-r\sigma^{(n)}} f(X_{\sigma^{(n)}}) - e^{-r(\sigma^{(n)} \wedge T^{(n)}_{k(n)})} f(X_{\sigma^{(n)}  \wedge T^{(n)}_{k(n)}}) \right| \right] \longrightarrow 0.
\end{equation}
Using this we indeed arrive at (\ref{tmp4}):

\begin{eqnarray}
 \liminf_{n \to \infty} \hat{V}_{k(n)}^{(n)}(x) & \stackrel{(\ref{def_hatV})}{\geq} &  \liminf_{n \to \infty} \mathbb{E}_x \left[ e^{-r(\sigma^{(n)} \wedge T^{(n)}_{k(n)})} f(X_{\sigma^{(n)} \wedge T^{(n)}_{k(n)}}) \right] \nonumber\\
 & \stackrel{(\ref{tmp3})}{=} & \liminf_{n \to \infty} \mathbb{E}_x \left[ e^{-r\sigma^{(n)}} f(X_{\sigma^{(n)}}) \right] \nonumber\\
 & \stackrel{(\ref{tmp2})}{=} & \mathbb{E}_x \left[ e^{-r\tau'} f(X_{\tau'}) \right] \nonumber\\
 & \stackrel{(\ref{tmp1})}{\geq} & v(T,x)-\eps, \nonumber
\end{eqnarray}
so let us prove (\ref{tmp2}) \& (\ref{tmp3}).

Ad (\ref{tmp2}): By construction of $\tau'$ and $\sigma^{(n)}$ we have

\[ \mathbb{E}_x \left[ e^{-r\tau'} f(X_{\tau'}) \right] - \mathbb{E}_x \left[ e^{-r\sigma^{(n)}} f(X_{\sigma^{(n)}}) \right] = \sum_{i=0}^N \mathbb{E}_x \left[ \mathbf{1}_{\{ \tau'=t_i \}} \left( e^{-r t_i} f(X_{t_i}) -  e^{-r\sigma^{(n)}_i} f(X_{\sigma^{(n)}_i}) \right) \right], \]
where 

\begin{equation}\label{tmp7}
 \sigma^{(n)}_i := \inf \left\{  T^{(n)}_k \, | \, T^{(n)}_k \geq t_i \right\}, 
\end{equation}
so it is enough to show that for each $i=0,\ldots,N$

\begin{equation}\label{tmp6} 
e^{-r\sigma^{(n)}_i} f(X_{\sigma^{(n)}_i}) \stackrel{L^1}{\longrightarrow} e^{-r t_i} f(X_{t_i}) \quad \mbox{as $n \to \infty$}. 
\end{equation}
Since $f$ is continuous and bounded, say by $M$, we have

\begin{eqnarray}
 \left| e^{-r\sigma^{(n)}_i} f(X_{\sigma^{(n)}_i}) - e^{-r t_i} f(X_{t_i}) \right| & \leq & \left| e^{-r\sigma^{(n)}_i} f(X_{\sigma^{(n)}_i}) - e^{-r t_i} f(X_{\sigma^{(n)}_i}) \right| \nonumber\\
  & & \quad + \left| e^{-r t_i} f(X_{\sigma^{(n)}_i}) - e^{-r t_i} f(X_{t_i}) \right| \nonumber\\
 & \leq & M \left( e^{-r\sigma^{(n)}_i}-e^{-r t_i} \right) + e^{-r t_i} \left|  f(X_{\sigma^{(n)}_i}) - f(X_{t_i}) \right| \nonumber 
\end{eqnarray}
so that by uniform integrability and the continuous mapping theorem we have that (\ref{tmp6}) follows from 

\begin{equation}\label{tmp8} 
\sigma^{(n)}_i \stackrel{\mathbb{P}_x}{\longrightarrow} t_i \quad \mbox{and} \quad X_{\sigma^{(n)}_i}  \stackrel{\mathbb{P}_x}{\longrightarrow} X_{t_i} \quad \mbox{as $n \to \infty$}. 
\end{equation}
Now, recalling (\ref{tmp7}) and that $T^{(n)}_k$ is a sum of i.i.d. exponentials with parameter $n$, due to the lack of memory property $\sigma^{(n)}_i-t_i$ is equal in distribution to an exponentially distributed random variable with mean $1/n$, independent of $X$, say $\xi^{(n)}$. So the first part of (\ref{tmp8}) is obvious and the second part is a direct consequence of stochastic continuity of $X$.

Ad (\ref{tmp3}): as above, it suffices to show

\begin{equation}\label{21nov1} 
\sigma^{(n)} - (\sigma^{(n)}  \wedge T^{(n)}_{k(n)}) \stackrel{\mathbb{P}_x}{\longrightarrow} 0 \quad \mbox{and} \quad X_{\sigma^{(n)}} - X_{\sigma^{(n)}  \wedge T^{(n)}_{k(n)}} \stackrel{\mathbb{P}_x}{\longrightarrow} 0 \quad \mbox{as $n \to \infty$}. 
\end{equation}
For the first one, for any $\delta>0$ we have

\begin{eqnarray}
\mathbb{P}_x \left( \left| \sigma^{(n)} - (\sigma^{(n)}  \wedge T^{(n)}_{k(n)}) \right| \geq \delta \right) & = & \mathbb{P}_x \left( \sigma^{(n)} \geq T^{(n)}_{k(n)} + \delta \right) \nonumber\\
& = & \mathbb{P}_x \left( \sigma^{(n)} \geq T^{(n)}_{k(n)} + \delta \quad \& \quad T^{(n)}_{k(n)} \leq T-\delta/2 \right) \nonumber\\
 & & \quad + \mathbb{P}_x \left( \sigma^{(n)} \geq T^{(n)}_{k(n)} + \delta \quad \& \quad T^{(n)}_{k(n)} > T-\delta/2 \right) \nonumber\\
& \leq & \mathbb{P}_x \left( T^{(n)}_{k(n)} \leq T-\delta/2 \right) + \mathbb{P}_x \left( \sigma^{(n)} \geq T + \delta/2 \right). \nonumber
\end{eqnarray}
The first probability on the last line vanishes as $n \to \infty$ on account of $T^{(n)}_{k(n)} \to T$ a.s. To see that the second probability vanishes as well, note that by construction $\tau' \leq T$ so that by definition of $\sigma^{(n)}$ the difference $\sigma^{(n)}-T$ is bounded above by $\inf \{ T^{(n)}_{k} \, | \, T^{(n)}_{k} \geq T \}-T$, but again by the lack of memory property this random variable follows an exponential distribution with mean $1/n$.

Finally, the second statement in (\ref{21nov1}) follows as above from the first statement and stochastic continuity of $X$.


\emph{Step 3.} Some standard manipulations, using H\"olders inequality and the boundedness of $f$, show that the definitions (\ref{def_hatV}) and (\ref{def_stop_new}) yield:

\begin{eqnarray}
\left| \hat{V}_{k(n)}^{(n)}(x) - V_{k(n)}^{(n)}(x) \right| & \leq & M \sup_{\tau \in \widetilde{\mathcal{T}}^{(n)}} \sum_{i=0}^{k(n)} \mathbb{E}_x \left[ \mathbf{1}_{\{ \tau= T^{(n)}_{i} \}} \left| e^{-ri/n} - e^{-r T^{(n)}_{i}} \right| \right] \nonumber\\
& \leq & M \sup_{\tau \in \widetilde{\mathcal{T}}^{(n)}} \sum_{i=0}^{k(n)} \mathbb{E}_x \left[ \mathbf{1}_{\{ \tau= T^{(n)}_{i} \}} \right]^{1/2} \mathbb{E}_x \left[ \left( e^{-ri/n} - e^{-r T^{(n)}_{i}} \right)^2 \right]^{1/2}. \label{27oct1}
\end{eqnarray} 
Recalling that $T^{(n)}_{i}$ is the sum of $i$ i.i.d. exponentials with mean $1/n$, it is straightforward to compute

\[ \mathbb{E}_x \left[ \left( e^{-ri/n} - e^{-r T^{(n)}_{i}} \right)^2 \right] = \frac{n^i}{(n+2r)^i} -2e^{-ri/n} \frac{n^i}{(n+r)^i} + e^{-2ri/n}. \]
As this expression is increasing in $i$ the supremum in (\ref{27oct1}) above is attained by $\tau=T^{(n)}_{k(n)}$. Plugging this in it is clear from the vanishing variance of $T^{(n)}_{k(n)}$ (or directly from the above display) that (\ref{27oct1}) indeed vanishes as $n \to \infty$.

\end{proof}

\begin{remark}\label{thm_main_ref} The assumption in the above Theorem \ref{thm_main} that $f$ is bounded and continuous keeps the proof compact. However the boundedness can be weakened to a more liberal integrability condition, and continuity can be weakened to (at least) a discrete set of discontinuities --- provided compound Poisson processes are excluded. Details are left to the reader. 
\end{remark}


\section{Example: the American put driven by a meromorphic L\'evy process}\label{sec_example}
\setcounter{equation}{0}

In this section we work out in detail our algorithm for an American put option driven by a meromorphic L\'evy process. This is probably the most classic example of an option in a Black \& Scholes type financial market. If the option has strike price $K>0$ and the price $S$ of the risky asset on which the option is written is assumed to evolve as $S_t=\exp(X_t)$ for $t \geq 0$ then the payoff function is given by

\begin{equation}\label{def_payoff_put}
f(x)=(K-e^x)^+.
\end{equation}

Let us first briefly discuss meromorphic L\'evy processes in more detail, see also \cite{Kuznetsov10}. A function $f: \mathbb{R}_{>0} \to \mathbb{R}$ is said to be discrete completely monotone if 

\begin{equation}\label{eq_dcm} 
f(x) = \sum_{m \geq 1} a_m e^{-b_m x} \quad \text{for all $x>0$,} 
\end{equation}
where $a_m > 0$, $b_m \geq 0$ for all $m \geq 1$ and $b_m \to \infty$ as $m \to \infty$. The definition is as follows.

\begin{definition}\label{def_MProc} A L\'evy process $X$ is meromorphic if the functions $x \mapsto \overline{\Pi}^+(x)=\Pi((x,\infty))$ and $x \mapsto \overline{\Pi}^-(x)=\Pi((-\infty,-x))$ are both discrete completely monotone.
\end{definition}
Alternatively, this is any L\'evy process whose L\'evy measure consists of an infinite mixture of exponentials on both the negative and positive halfline (cf. equation (5) in \cite{Kuznetsov10}). As mentioned in the Introduction, the key property of a meromorphic L\'evy process $X$  that makes our algorithm yield explicit expressions is the fact that the law of $X$ evaluated at an independent exponentially distributed random variable $\xi^{(q)}$ with mean $1/q$ is a mixture of exponentials:

\begin{equation}\label{dens_X} 
\mathbb{P} \left( X_{\xi^{(q)}} \in \mathrm{d}x \right) = \left( \mathbf{1}_{\{ x < 0 \}}  \sum_{j=0}^{N_{+}} c^{(q)}_{+}(j) e^{\zeta_{+}^{(q)}(j) x} + \mathbf{1}_{\{ x > 0 \}} \sum_{j=0}^{N_{-}} c^{(q)}_{-}(j) e^{\zeta_{-}^{(q)}(j) x} \right) \, \mathrm{d}x
\end{equation}
for $N_{-}, N_{+} \in \mathbb{N} \cup \{ \infty \}$ and sequences $(c^{(q)}_{-}(j)), (c^{(q)}_{+}(j)), (\zeta_{-}^{(q)}(j)), (\zeta_{+}^{(q)}(j))$, where for all (relevant) $j$: $c^{(q)}_{-}(j)>0$,  $c^{(q)}_{+}(j)>0$, $\zeta_{-}^{(q)}(j)<0$, $\zeta_{+}^{(q)}(j)>0$; and 

\[ \sum_{j \geq 0} \left( \frac{c^{(q)}_{+}(j)}{\zeta_{+}(j)} - \frac{c^{(q)}_{-}(j)}{\zeta_{-}(j)} \right)=1. \]
Note that $N_{-}= N_{+}=0$ corresponds to Brownian motion (with drift), while $N_{-}, N_{+} \in \mathbb{N}$ corresponds to a Brownian motion with drift plus (one- or two-sided) so-called hyper-exponential jumps (a generalisation of exponential jumps). See e.g. \cite{Kou02b} and \cite{Asmussen04}. For $N_{-}= N_{+}=\infty$ this expression for the density holds if and only if $X$ is a meromorphic L\'evy process (the proof relies heavily on the results in \cite{Kuznetsov10b}).

\begin{proposition}\label{prop_MProc} Let $X$ be a L\'evy process which is not compound Poisson. Then $X$ is a meromorphic L\'evy process if and only if $X$ satisfies (\ref{dens_X}) for some (and then all) $q>0$. Furthermore, if $X$ indeed satisfies (\ref{dens_X}), then $(\zeta_{-}^{(q)}(j))_{j \geq 0}$ (resp. $(\zeta_{+}^{(q)}(j))_{j \geq 0}$) is the sequence of all real negative (resp. real positive) zeros of the function  $z \mapsto q+\Psi(\mathrm{i}z)$ (or, equivalently, of $z \mapsto \Phi(-z)-q$), and  for all $j \geq 0$

\[ c^{(q)}_{+}(j) = \frac{q}{-\Phi'(-\zeta_{+}^{(q)}(j))}, \quad c^{(q)}_{-}(j) = \frac{q}{\Phi'(-\zeta_{-}^{(q)}(j))}. \]
Here $\Psi$ is the characteristic exponent of $X$ and $\Phi$ the Laplace exponent.
\end{proposition}

\begin{proof} 
The `only if' part can be found in Kuznetsov et al. \cite{Kuznetsov10}, Theorem 2. Let us now assume that $X$ satisfies (\ref{dens_X}) for some (and then all) $q>0$. By using Theorem 1 in Kuznetsov et al. \cite{Kuznetsov10}, it is sufficient to show that the functions $x \mapsto \mathbb{P}(X_{e(q)}>x)$ and $x \mapsto \mathbb{P}(X_{e(q)}<-x)$ are discrete completely monotone functions on $\mathbb{R}_{>0}$ (recall (\ref{eq_dcm})). Now
\[
\mathbb{P}(X_{\xi^{(q)}}>x)=\int \limits_{x}^{\infty} \sum_{j=0}^{N_{-}} c^{(q)}_{-}(j) \exp({\zeta_{-}^{(q)}(j) x}) \,\mathrm{d}x
=\sum_{j=0}^{N_{-}} - \frac{c^{(q)}_{-}(j)}{\zeta_{-}^{(q)}(j)} \exp(\zeta_{-}^{(q)}(j) x)
\]
If we now take $a_{j}=- \frac{c^{(q)}_{-}(j)}{\zeta_{-}^{(q)}(j)}>0$ and $b_{j}=-\zeta_{-}^{(q)}(j) > 0$ which is in fact increasing we see that indeed $x \mapsto \mathbb{P}(X_{e(q)}>x)$ is a discrete completely monotone function. This can be proven for $x \mapsto \mathbb{P}(X_{e(q)}<-x)$ in the same way.
\end{proof}

Let us briefly outline a few prominent examples of meromorphic L\'evy processes. The \emph{$\beta$-class} was introduced in \cite{Kuznetsov09}. It has a characteristic exponent given by

\begin{multline*} 
\Psi(z) = \frac{\sigma^2}{2} z^2 + \mathrm{i} az + \frac{c_1}{\beta_1} \left( B(\alpha_1,1-\lambda_1)-B(\alpha_1-\mathrm{i}z/\beta_1,1-\lambda_1) \right) \\
+ \frac{c_2}{\beta_2} \left( B(\alpha_2,1-\lambda_2)-B(\alpha_2+\mathrm{i}z/\beta_2,1-\lambda_2) \right) 
\end{multline*}
where $B(x,y)=\Gamma(x)\Gamma(y)/\Gamma(x+y)$ is the Beta function and the parameter ranges are $a \in \mathbb{R}$, $\sigma, c_i, \alpha_i, \beta_i>0$ and $\lambda_i \in (0,3) \setminus \{1,2\}$. The corresponding L\'evy measure has density

\[ \mathbf{1}_{\{ x<0 \}} c_2 \frac{e^{\alpha_2 \beta_2 x}}{(1-e^{\beta_2 x})^{\lambda_2}} + \mathbf{1}_{\{ x>0 \}} c_1 \frac{e^{-\alpha_1 \beta_1 x}}{(1-e^{-\beta_1 x})^{\lambda_1}}. \]
Note that setting $\sigma=0$ and $\lambda_i \in (0,2)$ yields a process with paths of bounded variation, all other choices paths of unbounded variation. Furthermore infinite (resp. finite) activity in the jump component can be obtained by setting $\lambda_i \in (1,3)$ (resp. $\lambda_i \not\in (1,3)$).

A second example is the \emph{$\theta$-class} introduced in \cite{Kuznetsov10b}. The L\'evy measure of such processes have a density of the form 

\[ \mathbf{1}_{\{ x<0 \}} c_2 \beta_2 e^{\alpha_2 x} \Theta_k(-x \beta_2) + \mathbf{1}_{\{ x>0 \}} c_1 \beta_1 e^{-\alpha_1 x} \Theta_k(x \beta_1) \]
where $\Theta_k$ is the $k$-th order (fractional) derivative of the theta function $x \mapsto \theta_3(0,e^{-x})$:

\[ \Theta_k(x) = \frac{\mathrm{d}^k}{\mathrm{d}x^k} \theta_3(0,e^{-x}) = \mathbf{1}_{\{ k=0 \}} + 2 \sum_{n \geq 1} n^{2k} e^{-n^2 x}, \quad x>0. \]
The parameter $\chi=k+1/2$ corresponds to the exponent of the singularity of the above density at $x=0$, and when $\chi \in \{ 1/2,3/2,5/2 \}$ (resp. $\chi \in \{ 1,2 \}$) the characteristic exponent is given in terms of trigonometric (resp. digamma) functions.

As a final example we mention \emph{general hypergeometric L\'evy processes}, these were introduced in \cite{Kyprianou09} as an example of how to use Vigon's theory of philantropy (cf. \cite{Vigon02}) for constructing new L\'evy processes. They have a characteristic exponent given by 

\[ \Psi(z) = \frac{\sigma^2}{2} z^2 + \mathrm{i} az + \Phi_1(-\mathrm{i}z) \Phi_2(\mathrm{i}z), \]
where $\sigma, a \in \mathbb{R}$ and $\Phi_i$ is the Laplace exponent of a (possibly killed) subordinator from the above mentioned $\beta$-class, which hence takes the form

\[ \Phi(z) = \kappa + \delta z + \frac{c}{\beta} \left( B(1-\alpha+\gamma,-\gamma)-B(1-\alpha+\gamma+z/\beta,-\gamma) \right) \]
for a killing rate $\kappa \geq 0$. 

These examples illustrate how rich the class of meromorphic L\'evy processes is. 

Let us now return to working out our algorithm in the case of the American put. We start by noting that for the payoff function (\ref{def_payoff_put}) it is straightforward to derive the following simplification of the recursion for the functions $(V^{(n)}_k)_{k \geq 0}$ stated in Theorem \ref{thm_main} (i).

\begin{proposition}\label{prop_alg_put} Suppose that $\mathbb{P}(X_1<0)>0$, i.e. $X$ is not a subordinator. Let $n \geq 1$. There exists a decreasing sequence of points $(\bar{x}^{(n)}_k)_{k \geq 0}$, with $\bar{x}^{(n)}_0:=\log K$, such that for each $k \geq 1$ we have:
\begin{itemize}
\item[(i)] the point $\bar{x}^{(n)}_k$ is the unique solution on $(-\infty,\bar{x}^{(n)}_{k-1})$ of the equation in $z$

\[ e^{-r/n} \mathbb{E}_z \left[  V_{k-1}^{(n)}(X_{\xi^{(n)}}) \right] -f(z)=0\]
\item[(ii)] $V^{(n)}_k$ is given by

\[ V^{(n)}_k(x) = \left\{ \begin{array}{ll}
f(x) & \mbox{if $x \leq \bar{x}^{(n)}_k$} \\
e^{-r/n} \mathbb{E}_x \left[  V_{k-1}^{(n)}(X_{\xi^{(n)}}) \right] & \mbox{if $x > \bar{x}^{(n)}_k$}
\end{array} \right. \]
\item[(iii)] $V^{(n)}_k>V^{(n)}_{k-1}$ on $(\bar{x}^{(n)}_k,\infty)$.
\end{itemize}
\end{proposition} 

\begin{proof} The proof is by induction. First we consider $k=1$. Denote for $s>0$

\[ F_1(s) =  e^{-r/n} \mathbb{E} \left[ \left( K-s e^{X_{\xi^{(n)}}} \right)^+ \right] - (K-s)^+ \]
(note that this is just the left hand side of the equation in (i) after the substitution $s=e^x$). Then $F_1(0+)=e^{-r/n}K-K<0$ and $F_1(\exp(\bar{x}^{(n)}_0))=F_1(K)>0$ as $X$ not being a subordinator implies that $\mathbb{P}(X_{\xi^{(n)}}<0)>0$. Furthermore, as $s \mapsto (K-s\alpha)^+$ is convex for any $\alpha$ the first term in the formula of $F_1$ is also convex and thus $F_1$ has a unique zero on $(0,K]$. Combining these observations with Theorem \ref{thm_main} (i) yields (i) and (ii) above. This also shows that $V^{(n)}_1(x)>(K-e^x)$ on $(\bar{x}^{(n)}_1,\bar{x}^{(n)}_0)$, while for  $x \in [\bar{x}^{(n)}_0,\infty)$ we have $V^{(n)}_1(x)>0=(K-e^x)^+$ as $X$ not being a subordinator implies $\mathbb{P}_x(X_{\xi^{(n)}}<\log K)>0$, so (iii) is also established.

Now suppose that (i)-(iii) above hold for $k-1$ (IH (i), IH (ii) and IH (iii)). Analogue to above, denote for $s>0$ the left hand side of the equation in (i) after the substitution $s=e^x$:

\[ F_k(s) =  e^{-r/n} \mathbb{E} \left[  V_{k-1}^{(n)}(\log s + X_{\xi^{(n)}})  \right] - f(\log s). \]
By IH (ii) we again have $F_k(0+)=e^{-r/n}K-K<0$. Also

\[ f(\bar{x}^{(n)}_{k-1}) = e^{-r/n} \mathbb{E}_{\bar{x}^{(n)}_{k-1}} \left[  V_{k-2}^{(n)}(X_{\xi^{(n)}}) \right] < e^{-r/n} \mathbb{E}_{\bar{x}^{(n)}_{k-1}} \left[  V_{k-1}^{(n)}(X_{\xi^{(n)}}) \right], \]
the equality by IH (i) and the inequality by IH (iii) together with $X$ not being a subordinator. This translates to $F_k(\exp(\bar{x}^{(n)}_{k-1}))>0$. As above, the first term in the formula of $F_k$ is convex since $s \mapsto V_{k-1}^{(n)}(\log s + \alpha)$ is convex for any $\alpha$, as is clear from the definition (\ref{def_stop_new}). Consequently $F_k$ has a unique zero on $(0,K]$, located in $(-\infty,\bar{x}^{(n)}_{k-1})$, and in combination with Theorem \ref{thm_main} (i), statement (i) and (ii) follow. 

It remains to show (iii). It follows from the above and IH (ii) that $V^{(n)}_k(x)>V^{(n)}_{k-1}(x)=(K-e^x)$ on $(\bar{x}^{(n)}_k,\bar{x}^{(n)}_{k-1}]$. For $x > \bar{x}^{(n)}_{k-1}$ we have by (ii), IH (ii), IH (iii) and $X$ not being a subordinator that

\[  V^{(n)}_k(x) = e^{-r/n} \mathbb{E}_{x} \left[  V_{k-1}^{(n)}(X_{\xi^{(n)}}) \right] > e^{-r/n} \mathbb{E}_{x} \left[  V_{k-2}^{(n)}(X_{\xi^{(n)}}) \right] = V^{(n)}_{k-1}(x) \]
and we are done.
\end{proof}

Note that the above result shows that in the approximating value functions it is optimal to stop as soon as $X$ falls below a certain level. Or, equivalently, as soon as the stock price $S=\exp(X)$ falls below a certain level. This is structurally consistent with the exact optimal exercise strategy in the original problem, which is to stop as soon as $S$ falls below a certain time dependent boundary (cf. e.g. \cite{Karat01}). 

If $X$ satisfies (\ref{dens_X}), i.e. in particular if $X$ is a meromorphic L\'evy process, then it turns out the expectations in the above Proposition \ref{prop_alg_put} can be worked out explicitly. The proof uses induction over $k$ and consists of tedious yet straightforward computations, it is therefore omitted.

\begin{proposition}\label{prop_put_form} Suppose that $X$ satisfies (\ref{dens_X}) and is not compound Poisson. Fix some $n \geq 1$. For any $k \geq 0$ the function $V^{(n)}_k$ can be expressed in piecewise form as follows:

\begin{equation}\label{formula_Vk}
 \begin{split}
V^{(n)}_k(x) = {} & \mathbf{1}_{\{ x \geq \bar{x}^{(n)}_0 \}}  \sum_{j=0}^{N_{+}} e^{-\zeta_{+}^{(n)}(j) x}  \sum_{i=0}^{k-1} A^{(n)}_{+}(i,j,0,k) x^i \\
& + \sum_{m=1}^k  \mathbf{1}_{\{ x \in [\bar{x}^{(n)}_{m},\bar{x}^{(n)}_{m-1}) \}} \Biggl( \sum_{j=0}^{N_{-}} e^{-\zeta_{-}^{(n)}(j) x} \sum_{i=0}^{k-m} A^{(n)}_{-}(i,j,m,k) x^i \\
& \qquad + \sum_{j=0}^{N_{+}} e^{-\zeta_{+}^{(n)}(j) x} \sum_{i=0}^{k-m} A^{(n)}_{+}(i,j,m,k) x^i - B^{(n)}(m,k) e^x + C^{(n)}(m,k) \Biggr) \\
& + \mathbf{1}_{\{ x<\bar{x}^{(n)}_k \}} (K-e^x),
\end{split}
\end{equation}
where $\bar{x}^{(n)}_0=\log K$. Expressions for $\bar{x}^{(n)}_k$ and all the coefficients involved can be found in Appendix \ref{sec_app_form}.
\end{proposition}

\begin{remark}\label{rem1}
\begin{itemize}
 \item[(i)] Note that in the Brownian motion case, the structure of the formulae for the $V^{(n)}_k$'s from Proposition \ref{prop_put_form} matches the corresponding formulae found by Carr with his Canadisation method (cf. Section \ref{sec_alg}). That is, the formulae coincide except for the values of the coefficients $A$, $B$ and $C$, and for the values of the boundary points $\bar{x}$.
\item[(ii)] Any payoff function $f$ that can (piecewise) be expressed as a linear combination of functions of the form $A x^i e^{B x} + C$ for $A, B, C \in \mathbb{R}$ and $i \in \mathbb{N}$ can be dealt with in similar fashion and yields a similar structure as (\ref{formula_Vk}). If a result as in Proposition \ref{prop_alg_put} is not available, i.e. if it is a priori not clear how the maximum in the recursive relation (\ref{main_rec}) works out, a way forward is to implement a dynamic decision rule where first a formula for 

\[ F_k(x) = e^{-r/n} \mathbb{E}_{x} \left[  V_{k-1}^{(n)}(X_{\xi^{(n)}}) \right] \]
is computed (which is still of the same structure as in (\ref{formula_Vk})) and then $V_k=\max \{ F_k, f \}$ can be determined in a straightforward way. 

Finally, a payoff function that is not of this convenient form mentioned above can easily be approximated by a function which is of this convenient form, and it is straightforward to estimate the error due to this approximation.
\item[(iii)] It would be possible to further extend this approach to deal with more complicated payoffs involving path dependency, for instance to cases where the payoff at time $t$ depends also on the running supremum $\overline{X}_t = \sup_{s \leq t} X_s$ as is the case for `lookback options'. A step in the algorithm would then not only require $X_{\xi^{(q)}}$ but the pair $(X_{\xi^{(q)}},\overline{X}_{\xi^{(q)}})$ rather. Explicit expressions for the law of this pair is available for meromorphic L\'evy processes.
\end{itemize}
\end{remark}

\section{Some numerics for the American put}\label{sec_num}

In this section we discuss some numerics for the American put, obtained by a computer implementation of the result in Proposition \ref{prop_put_form}. In particular we present some graphs of approximating value functions $V^{(n)}_k$ and some graphs of the approximative optimal exercise boundary. For the latter, as is well known (cf. e.g. \cite{Karat01}) the optimal stopping time $\tau^*$ in (\ref{def_stop_orig}) with $f(x)=(K-e^x)^+$ is of the form

\begin{equation}\label{opt_time_Am_put} 
\tau^* = \inf \{ u \geq 0 \, | \, v(T-u,X_u) = f(X_u) \} = \inf \{ u \geq 0 \, | \, X_u \leq b(T-u) \} 
\end{equation}
for some non-increasing function $b$ on $\mathbb{R}_{>0}$ referred to as the optimal exercise boundary, which may hence be defined as

\[ b(t) := \sup \{ y \in \mathbb{R} \, | \, v(t,y)=f(y) \} \leq \log (K) \quad \text{for $t>0$.} \]
Hence, since for large $n$ we have $V^{(n)}_k(\,.\, ) \approx v(k/n,\,.\,)$ for all $k \geq 1$ and recalling the result of Proposition \ref{prop_alg_put} it should be expected\footnote{We do not provide a rigorous proof of this, which seems slightly tricky due to the fact that the principle of smooth pasting dictates that at least when the driving L\'evy process has a Gaussian part $v$ `pastes smoothly' onto $f$, i.e. their first derivatives in the space component coincide. Consequently it is not a straightforward application of the Implicit function theorem} that the set of points $\{ (k/n,\bar{x}^{(n)}_k) \}_{k \geq 1}$, in the captions of the plots referred to as `boundary points', form an approximation of the set $\{ (k/n,b(k/n)) \}_{k \geq 1}$.

Next, we briefly describe the two main tasks for creating the graphs of different value functions and boundary points presented further below. 

The first step is the calculation of the solutions $\zeta_{+}^{(q)}(i)$ and $\zeta_{-}^{(q)}(j)$ of $z \mapsto q+\Psi(\mathrm{i}z)$ and of the coefficients $c^{(q)}_{+}(i)$ and $c^{(q)}_{-}(j)$ by using Proposition \ref{prop_MProc} for $i \in \left\{0,\ldots, N_{+}\right\}$ and $j \in \left\{0,\ldots, N_{-}\right\}$. In case a meromorphic L\'evy process is used for which $N_{+}=\infty$ or  $N_{-}=\infty$ the infinite sums need to be truncated. For this we used the ad-hoc approach of ensuring that the truncated sums represent at least $99\%$ of the mass of $X_{\xi^{(n)}}$, i.e. (compare with (\ref{dens_X})) we determined $M^+, M^-$ such that

\[ \sum_{j=0}^{M^+} \frac{c^{(n)}_{+}(j)}{\zeta_{+}(j)} - \sum_{j=0}^{M^-} \frac{c^{(n)}_{-}(j)}{\zeta_{-}(j)} \geq 0.99 \]
and then normalised the coefficients. It is worth noting that $M$ is larger in the case of meromorphic L\'evy process with finite variation than in the cases with infinite variation. Particularly, if $\sigma>0$ the largest part of the probability mass is concentrated on $\zeta_{+}(0)$ and $\zeta_{-}(0)$ whereas the other zeroes contribute only a small part.

The second main step is the recursive computation of the value function and the boundary points using formula (\ref{formula_Vk}). A program for this purpose was written in the language C++. We found that in certain cases the algorithm may need more than the default `double precision' as working precision in order to have enough significant digits left in the end result. For this we made use of the `GMP/MPFR/MPFRC++' packages. Furthermore we made use of `OpenMP' to parallelise the operations. 

The L\'evy processes used in this section are Brownian motion with drift, Brownian motion with drift plus exponentially distributed jumps (Kou model) and the $\beta$-class, see Section \ref{sec_example}.


The Figures \ref{fig:lambdavalue} and \ref{fig:lambdabound} contain different value functions and its boundary points for processes of the $\beta$-class. We compare different values of $\lambda_{1}=\lambda_{2}$ which are responsible for the small jump behaviour. One process is similar to a CGMY process ($\lambda_{1}=1.8$, finite variation and infinite activity, see \cite{carr02}). The other processes have finite variation and finite activity ($\lambda_{1}=0.95$) respectively infinite variation and infinite activity ($\lambda_{1}=2.1$). \\

\begin{figure}[H]
\centering
\centerline{\includegraphics[width=1.2\textwidth]{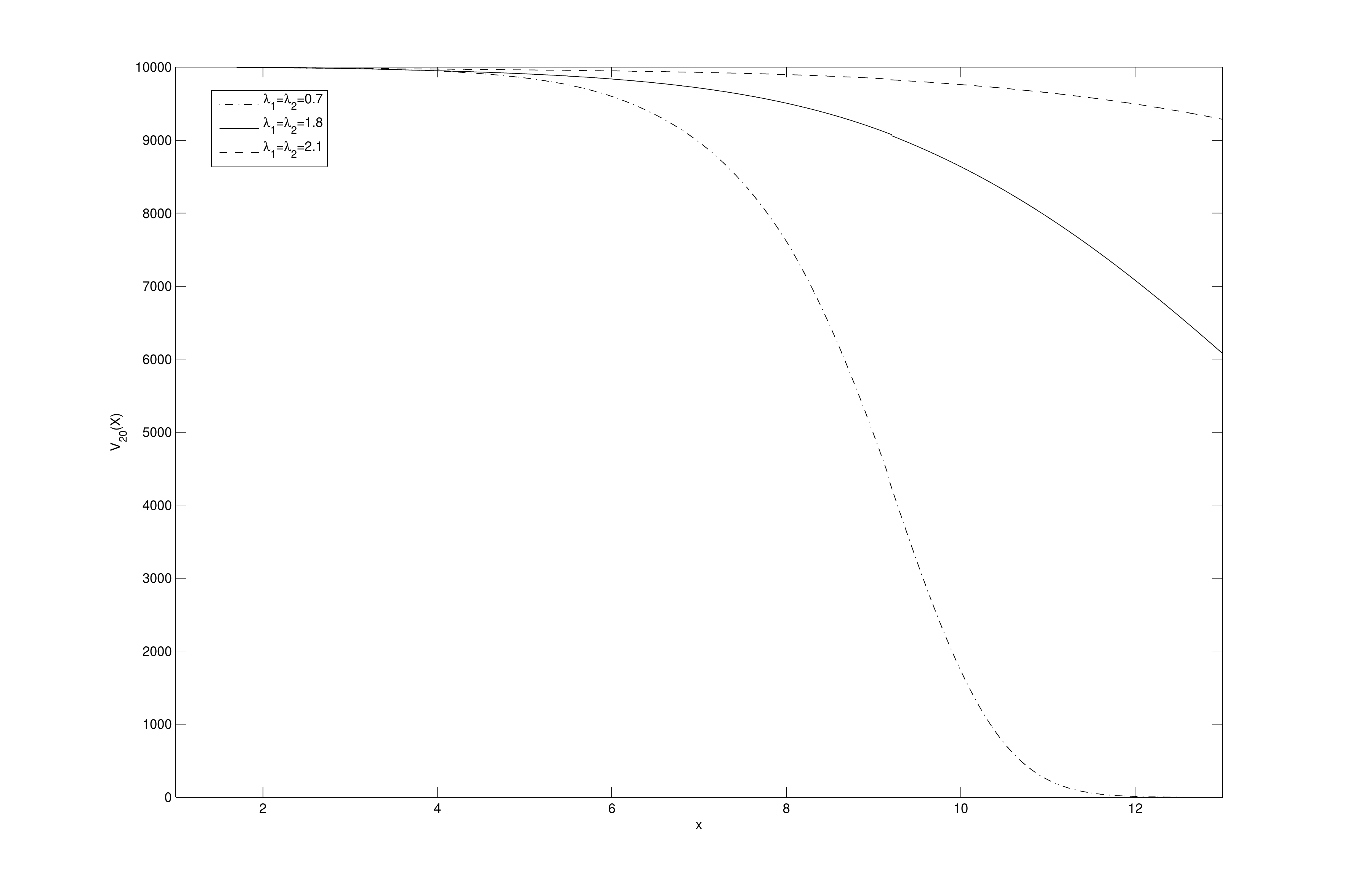}}
\caption{Plots of the approximating value function $V^{(n)}_k$ where the driving process $X$ is of the $\beta$-class with parameter values $\sigma=0, \alpha_{1}=10, \alpha_{2}=100, \beta_{1}=\beta_{2}=0.5, c_{1}=c_{2}=100$ and $\lambda_{1}=\lambda_{2}=0.95, 1.8, 2.1$. The drift $\mu$ is chosen as the solution of $\Psi(-I)=-r$. Furthermore $n=100, k=20, r=0.05, K=10000$ and $M^+=31, M^-=61$}
\label{fig:lambdavalue}
\end{figure}

\begin{figure}[H]
\centering
\centerline{\includegraphics[width=1.2\textwidth]{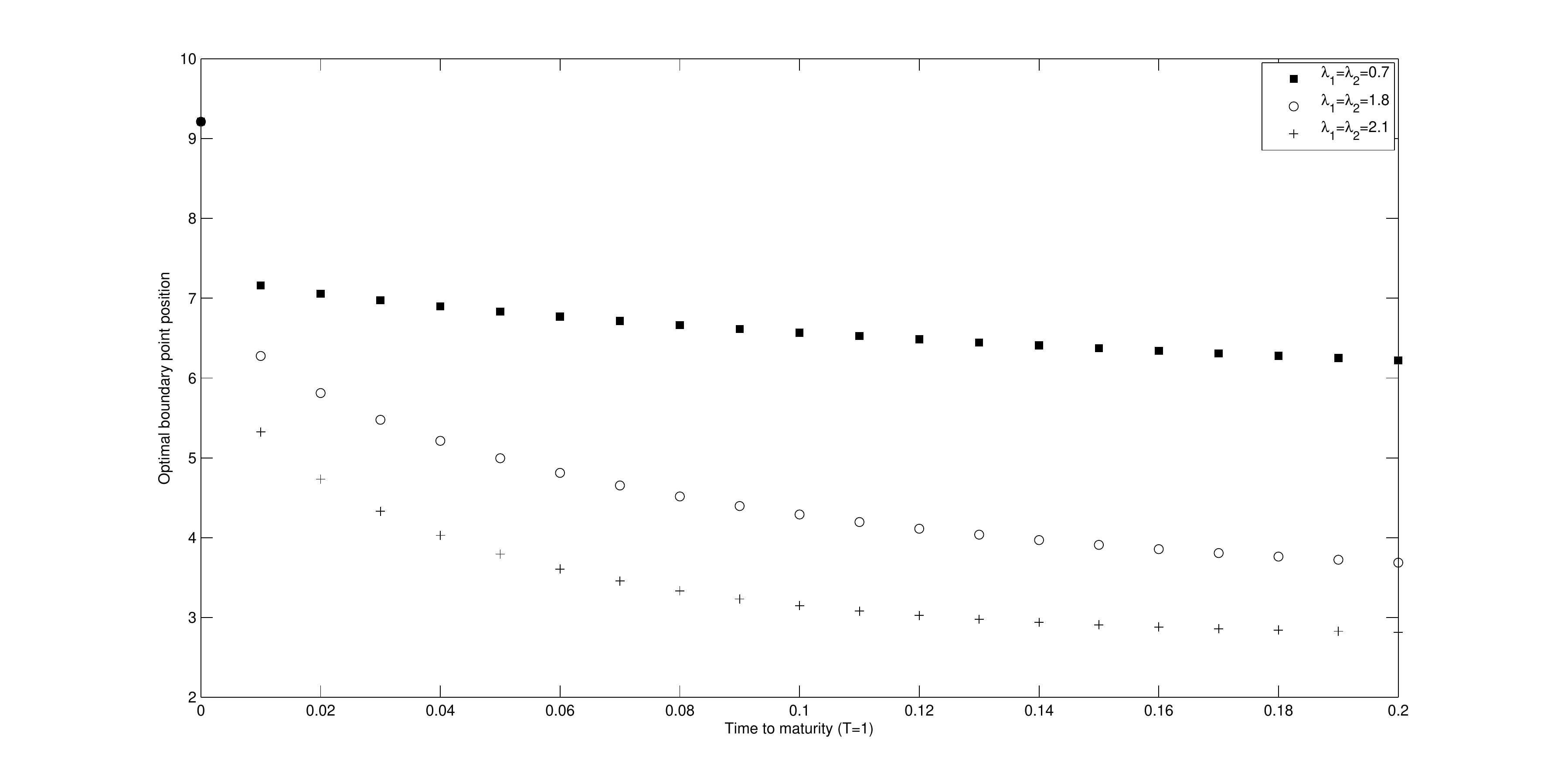}}
\caption{Three sets of boundary points obtained using the same parameter values and settings as in Figure \ref{fig:lambdavalue} 
}
\label{fig:lambdabound}
\end{figure}

In Figures \ref{fig:betasigmacompvalue} and \ref{fig:betasigmacompbound} the value functions and boundary points for a process of the $\beta$-class with infinite variation and infinite activity are plotted for different values of the Gaussian coefficient $\sigma$. 
If we compare over certain intervals $(\bar{x}^{(q)}_{i}, \bar{x}^{(q)}_{j})$, $i>j$ the curves of the value functions and boundary points for different $\sigma_{1}$ and $\sigma_{2}$, $\sigma_{1}, \sigma_{2} \in \{0,1,2,6\}$ we observe the following. If the curve of the value function for $\sigma_{1}$ is steeper on $(\bar{x}^{(q)}_{i}, \bar{x}^{(q)}_{j})$ than the one for a different $\sigma_{2}$, then the curve of the boundary points for $\sigma_{1}$ appears to be more flat on that interval and vice versa. This is in accordance to general optimal stopping theory. 
Note that for $\sigma=0$ there seems to be a jump in the boundary at $0$, that is Figure \ref{fig:betasigmacompbound} indicates that $b(0+)<\log K$ if $\sigma=0$ while $b(0+)=\log K$ if $\sigma>0$. This is consistent with the observation in \cite{lamberton2011}.

\begin{figure}[H]
\centering
\centerline{\includegraphics[width=1.2\textwidth]{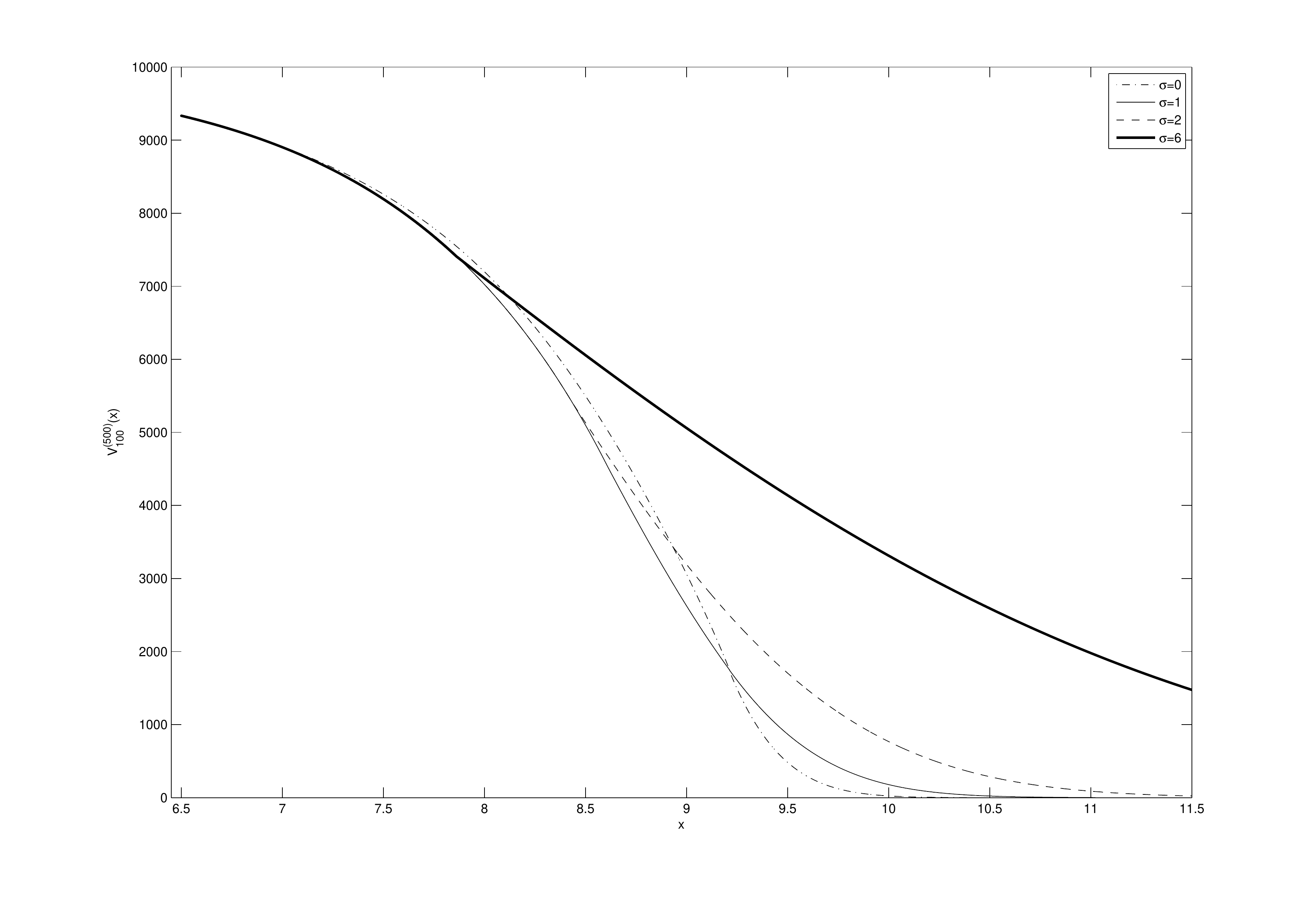}}
\caption{Plots of the approximating value function $V^{(n)}_k$ where the driving process $X$ is of the $\beta$-class with infinite variation and infinite activity. Parameter values are $\sigma=0 ,1 ,2 ,6,\alpha_{1}=56, \alpha_{2}=56.4, \beta_{1}=\beta_{2}=2, c_{1}=c_{2}=1.5$ and $\lambda_{1}=\lambda_{2}=2.8$. Furthermore $n=500, k=100, r=0.05, K=10000$ and $M^-=M^+=1$}
\label{fig:betasigmacompvalue}
\end{figure}

\begin{figure}[H]
\centering
\centerline{\includegraphics[width=1.2\textwidth]{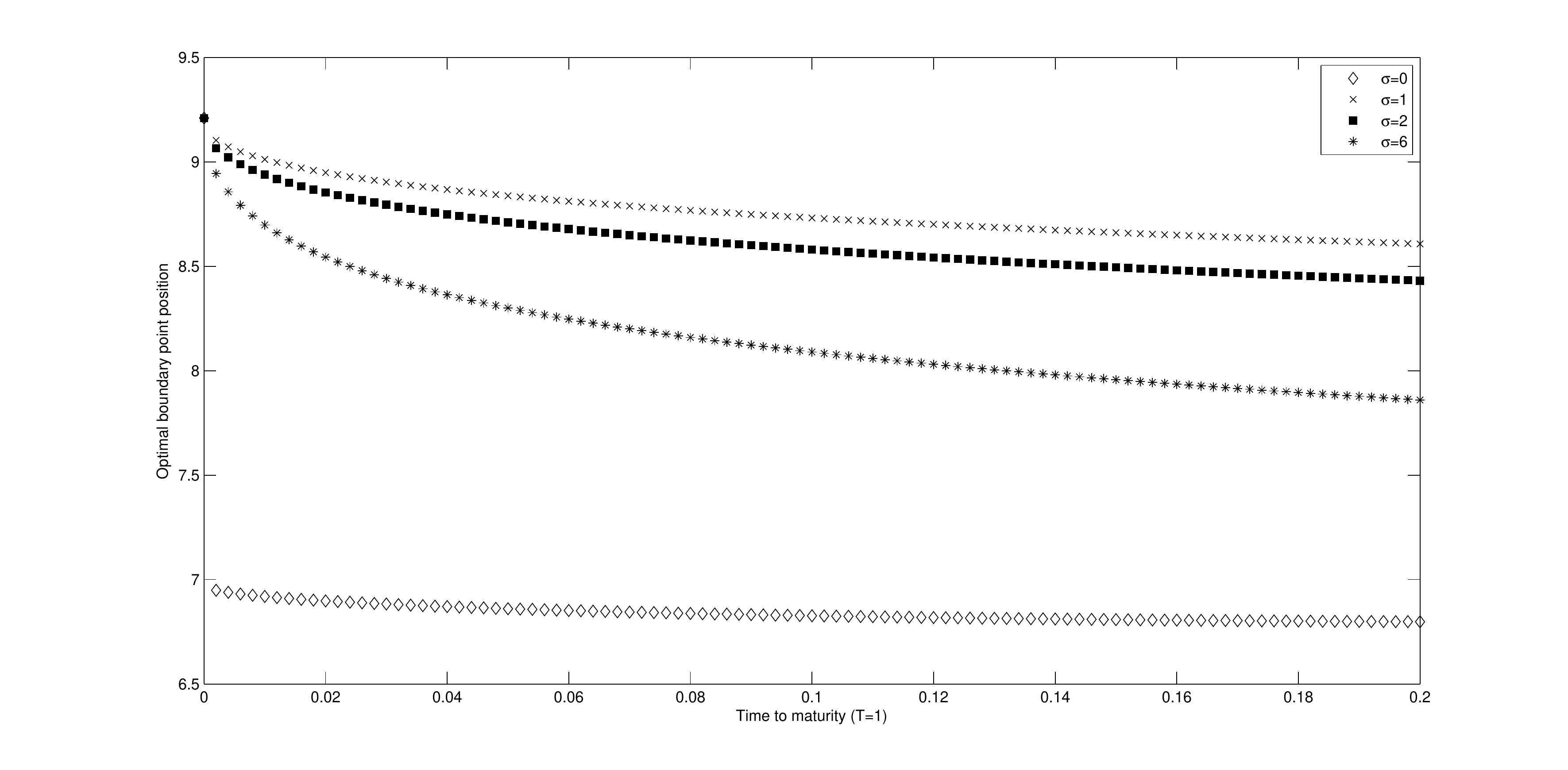}}
\caption{Four sets of boundary points obtained using the same parameter values and settings as in Figure \ref{fig:betasigmacompvalue} 
}
\label{fig:betasigmacompbound}
\end{figure}

Finally, Figure \ref{fig:relativedifference} illustrates the convergence. Hereby we compare the relative difference between the value functions for several $n$ and $k$ with a constant ratio $k/n=0.2$.

\begin{figure}[H]
\centering
\centerline{\includegraphics[width=1.2\textwidth]{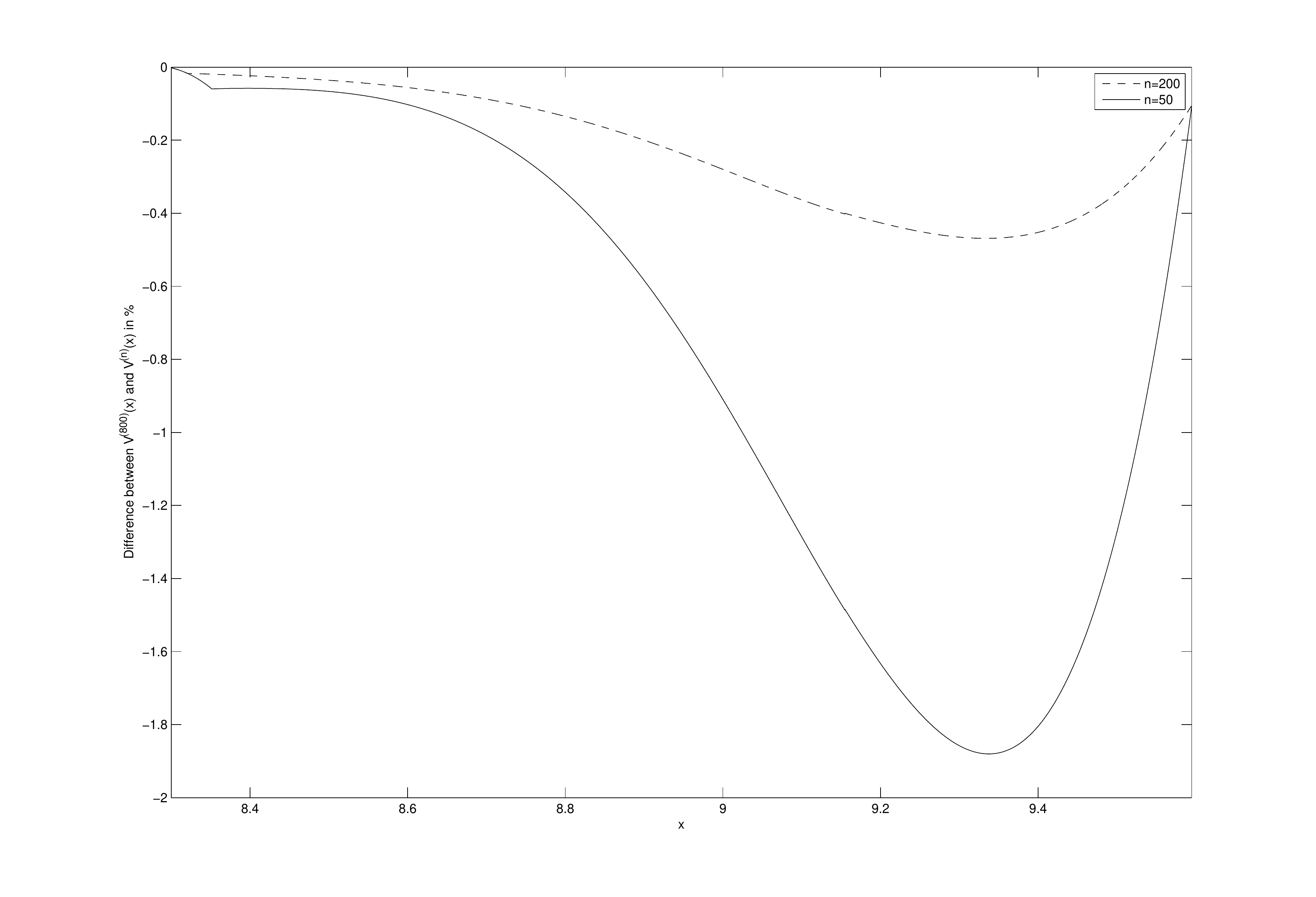}}
\caption{Plots of the relative difference between three approximating value functions, for $n=800$ and $k=160$; $n=200$ and $k=40$; $n=50$ and $k=10$. The plots are `zoomed in' to the part of the $x$-axis where the difference is largest. The case $n=800$ and $k=160$ is taken as the base. The underlying $\beta$-class process with infinite variation and infinite activity has the following parameters: $\sigma=1, \alpha_{1}=1, \alpha_{2}=1, \beta_{1}=\beta_{2}=80, c_{1}=c_{2}=1$ and $\lambda_{1}=\lambda_{2}=2.3$. The drift $\mu$ is chosen as the solution of $\Psi(-I)=-r$.  Furthermore $r=0.05, K=10000$ and $M^+=M^-=1$} 
\label{fig:relativedifference}
\end{figure}

\newpage

\section*{Acknowledgements}

The authors are very grateful to Alexey Kuznetsov, Andreas Kyprianou and Antonio Murillo Salas for very useful discussions \& suggestions, and to the IT services team at the University of Manchester for their support in computer questions.

\appendix
\section{Formulae for the coefficients appearing in Proposition \ref{prop_put_form}}\label{sec_app_form}
\renewcommand{\theequation}{A.\arabic{equation}}
\setcounter{equation}{0}

This Appendix contains expressions for the coefficients that appear in Proposition \ref{prop_put_form}, recursively in $k$ for some $n \geq 1$ fixed. Recall that the starting point for the recursion is (\ref{formula_Vk}) for $k=0$, which contains only one coefficient, namely $\bar{x}^{(n)}_{0}:=\log K$. Throughout the sequel, fix some $k \geq 1$. 

Let the intervals $I_{l}^{(k)}=[\bar{x}^{(n)}_{l},\bar{x}^{(n)}_{l-1})$ for $l=1,\ldots,k-1$, $I_0=[\bar{x}^{(n)}_{0},\infty)$ and $I_k=(-\infty,\bar{x}^{(n)}_{k-1})$. Furthermore let $F_+(x;l,j,k)$ be an antiderivative of $x \mapsto V^{(n)}_{k-1}(x) e^{\zeta_{+}^{(n)}(j) x}$ on the interval $I_{l}^{(k)}$ and $F_-(x;l,j,k)$ an antiderivative of $x \mapsto V^{(n)}_{k-1}(x) e^{\zeta_{-}^{(n)}(j) x}$ on the same interval $I_{l}^{(k)}$. Recalling the formula of $V^{(n)}_{k-1}$ from (\ref{formula_Vk}), note that $F_+(\argdot;l,j,k)$ and $F_-(\argdot;l,j,k)$ consist of sums of antiderivatives of functions of the form $x \mapsto x^i e^{bx}$ for $i \in \mathbb{N}$, it is hence straightforward to apply partial integration and express them in terms of elementary functions. For compactness we do not spell out these formulae in all detail here.

Furthermore, define

\begin{eqnarray}
\Upsilon_0 & = & \sum_{q=0}^{N_+} \frac{c^{(n)}_{+}(q)}{\zeta_{+}^{(n)}(q)+1} - \sum_{q=0}^{N_-} \frac{c^{(n)}_{-}(q)}{\zeta_{-}^{(n)}(q)+1}, \nonumber\\
\Upsilon^{(n)}_1(h,j) & = & (-1)^{h} \Bigg( \sum_{q=0,\,q \not=j}^{N_+} c^{(n)}_{+}(q) \left( \zeta_{+}^{(n)}(q)-\zeta_{+}^{(n)}(j) \right)^{-(h+1)} \nonumber\\
& & \quad - \sum_{q=0}^{N_-} c^{(n)}_{-}(q) \left( \zeta_{-}^{(n)}(q)-\zeta_{+}^{(n)}(j) \right)^{-(h+1)} \Bigg), \nonumber\\
\Upsilon^{(n)}_2(h,j) & = & (-1)^{h} \Bigg( \sum_{q=0}^{N_+} c^{(n)}_{+}(q) \left( \zeta_{+}^{(n)}(q)-\zeta_{-}^{(n)}(j) \right)^{-(h+1)} \nonumber\\
& & \quad - \sum_{q=0,\, q\not= j}^{N_-} c^{(n)}_{-}(q) \left( \zeta_{-}^{(n)}(q)-\zeta_{-}^{(n)}(j) \right)^{-(h+1)} \Bigg), \nonumber
\end{eqnarray}
where $h \in \mathbb{N}$, $j \in \{0,\ldots,N_+\}$ for $\Upsilon^{(n)}_1$ and $j \in \{0,\ldots,N_-\}$ for $\Upsilon^{(n)}_2$.

The coefficients appearing in the formula (\ref{formula_Vk}) for $V^{(n)}_{k}$ can now be expressed as follows -- note that we as usual understand an empty sum as equal to zero, and that the expressions for $A^{(n)}_{+}(\argdot,j,\argdot,\argdot)$ (resp. $A^{(n)}_{-}(\argdot,j,\argdot,\argdot)$) below are valid for all $j \in \{0,\ldots,N_+\}$ (resp. $j \in \{0,\ldots,N_-\}$):\newline
i) On the interval $[\bar{x}^{(n)}_{k},\bar{x}^{(n)}_{k-1})$:

\begin{eqnarray}
A^{(n)}_{+}(0,j,k,k) & = & 0; \nonumber\\
A^{(n)}_{-}(0,j,k,k) & = & e^{-r/n} c^{(n)}_{-}(j) \Bigl( F_-(\bar{x}^{(n)}_{k-1};k,j,k)-F_-(\bar{x}^{(n)}_{0};0,j,k) \nonumber\\
 & & + \sum_{l=1}^{k-1} \left( F_-(\bar{x}^{(n)}_{l-1};l,j,k)-F_-(\bar{x}^{(n)}_{l};l,j,k) \right) \Bigr); \nonumber\\
B^{(n)}(k,k) & = & e^{-r/n} \Upsilon_0; \nonumber\\
C^{(n)}(k,k) & = & e^{-r/n} K, \nonumber 
\end{eqnarray}
where $\bar{x}^{(n)}_{k}$ is the unique solution on $(-\infty,\bar{x}^{(n)}_{k-1})$ of the equation in $z$:

\[ \sum_{q=0}^{N_{-}} e^{-\zeta_{-}^{(n)}(q) z} A^{(n)}_{-}(0,q,k,k) - B^{(n)}(k,k) e^z + C^{(n)}(k,k) = K-e^z. \]

\noindent
ii) On the interval $[\bar{x}^{(n)}_{m},\bar{x}^{(n)}_{m-1})$ for $m=1,\ldots,k-1$:

\begin{eqnarray}
A^{(n)}_{+}(0,j,m,k) & = & e^{-r/n} c^{(n)}_{+}(j) \Biggl( F_+(\bar{x}^{(n)}_{k-1};k,j,k)-F_+(\bar{x}^{(n)}_{m};m,j,k) \nonumber\\
  &  & \qquad +\sum_{l=m+1}^{k-1} \left( F_+(\bar{x}^{(n)}_{l-1};l,j,k)-F_+(\bar{x}^{(n)}_{l};l,j,k) \right) \Biggr) \nonumber\\
  &  & + \, e^{-r/n} \sum_{h=0}^{k-m-1} A^{(n)}_{+}(h,j,m,k-1)  \Upsilon^{(n)}_1(h,j) h!; \nonumber\\
A^{(n)}_{+}(i,j,m,k) & = & e^{-r/n}  \sum_{h=1}^{k-m-i} A^{(n)}_{+}(i+h-1,j,m,k-1)  \Upsilon^{(n)}_1(h-1,j) \frac{(h+i-1)!}{i!} \nonumber\\
 & & + \, e^{-r/n} \frac{c^{(n)}_{+}(j)}{i} A^{(n)}_{+}(i-1,j,m,k-1), \quad \forall i \in \{ 1,\ldots,k-m \}; \nonumber\\
 A^{(n)}_{-}(0,j,m,k) & = & e^{-r/n} c^{(n)}_{-}(j) \Biggl( F_-(\bar{x}^{(n)}_{m-1};m,j,k)-F_-(\bar{x}^{(n)}_{0};0,j,k) \nonumber\\
  &  & \qquad +\sum_{l=1}^{m-1} \left( F_-(\bar{x}^{(n)}_{l-1};l,j,k)-F_-(\bar{x}^{(n)}_{l};l,j,k) \right) \Biggr) \nonumber\\
  &  & + \, e^{-r/n} \sum_{h=0}^{k-m-1} A^{(n)}_{-}(h,j,m,k-1)  \Upsilon^{(n)}_2(h,j) h!; \nonumber\\
A^{(n)}_{-}(i,j,m,k) & = & e^{-r/n}  \sum_{h=1}^{k-m-i} A^{(n)}_{-}(i-1+h,j,m,k-1)  \Upsilon^{(n)}_2(h-1,j) \frac{(h+i-1)!}{i!} \nonumber\\
 & & - \, e^{-r/n} \frac{c^{(n)}_{-}(j)}{i} A^{(n)}_{-}(i-1,j,m,k-1), \quad \forall i \in \{ 1,\ldots,k-m \}; \nonumber\\
B^{(n)}(m,k) & = & e^{-r/n} \Upsilon_0 B^{(n)}(m,k-1); \nonumber\\
C^{(n)}(m,k) & = & e^{-r/n} C^{(n)}(m,k-1). \nonumber  
\end{eqnarray}

\noindent
iii) On the interval $[\bar{x}^{(n)}_{0},\infty]$:

\begin{eqnarray}
A^{(n)}_{+}(0,j,0,k) & = & e^{-r/n} c^{(n)}_{+}(j) \Biggl( F_+(\bar{x}^{(n)}_{k-1};k,j,k)-F_+(\bar{x}^{(n)}_{0};0,j,k) \nonumber\\
  &  & \qquad +\sum_{l=1}^{k-1} \left( F_+(\bar{x}^{(n)}_{l-1};l,j,k)-F_+(\bar{x}^{(n)}_{l};l,j,k) \right) \Biggr) \nonumber\\
  &  & + \, e^{-r/n} \sum_{h=1}^{k-1} A^{(n)}_{+}(h-1,j,0,k-1)  \Upsilon^{(n)}_1(h-1,j) (h-1)!; \nonumber\\
A^{(n)}_{+}(i,j,0,k) & = & e^{-r/n}  \sum_{h=1}^{k-1-i} A^{(n)}_{+}(i+h-1,j,0,k-1)  \Upsilon^{(n)}_1(h-1,j) \frac{(h+i-1)!}{i!} \nonumber\\
 & & + \, e^{-r/n} \frac{c^{(n)}_{+}(j)}{i} A^{(n)}_{+}(i-1,j,0,k-1), \quad \forall i \in \{ 1,\ldots,k-1 \}. \nonumber
\end{eqnarray}

\end{document}